\documentclass[11pt]{article}
\usepackage{amsmath}
\usepackage{amssymb}
\usepackage{amsthm}
\usepackage{amscd}
\usepackage[all]{xy}
\usepackage{graphicx}
\usepackage{color}
\usepackage{subcaption,mathtools}
\numberwithin{equation}{section}

\newcounter{AbcT}

\newtheorem {Theorem}    {Theorem}[section]

\newtheorem {Subword}    {Subpath}

\newtheorem {Lemma}      [Theorem]    {Lemma}
\newtheorem {Corollary}   [Theorem] {Corollary}
\newtheorem {Proposition}[Theorem]    {Proposition}

\theoremstyle{remark}

\newcommand{{\rsp}}{respecting}

\newcommand{\Cay}{\mathit{Cay}}

\newcounter{DM@bibnum}

\newcommand {\SL} {{\mathrm{SL}}}

\newcommand {\N} {{\mathcal N}}

\newcommand {\Z} {{\mathbb Z}}

\def\NSL_2{{\mathcal N SL_2}}

\def\phi{\varphi}

\def\hbar{\bar h}

\newcommand{\iv}{^{-1}}

\begin{document}

\title{Divergence of Thompson groups}
\author{Gili Golan, Mark Sapir\thanks{The research of the first author was supported in part by a Fulbright grant and a post-doctoral scholarship of Bar-Ilan University, the research of the second author was supported in part by the NSF grant DMS-1500180.}}
\date{}

\maketitle

\section{Introduction and Preliminaries}

R. Thompson groups $F, T, V$, introduced by R. Thompson about 50 years ago, are very interesting from many points of view and there are numerous recent papers devoted to these groups (see \cite{BCST, BS, BBGGHMS, BCS,  J17, GS, G} and many others). Conferences devoted exclusively to these groups happen every few years. These are groups of certain homeomorphisms of the Cantor set, but can be best described as groups of tree diagrams (or \emph{braided pictures} \cite{GubSa}). Namely, an element $g$ of $V$ is a triple $(T_+, \sigma, T_-)$ where $T_+$ and $T_-$ are finite (full) rooted binary trees drawn on the plane with the same set $L$ of leaves (ordered from left to right) and $\sigma$ is a permutation of the set $L$. Every leaf $o$ of $T_+$ or $T_-$ corresponds to a unique path from the root to the leaf. For every caret of a plane full binary tree, we can label its left edge by 0 and the right edge by 1.
So every leaf $o$ of such a tree $T$ corresponds to a binary word $\ell_o(T)$ labeling the path from the root to $o$. We shall not distinguish between the path on the tree connecting the root with the leaf $o$ and the word $\ell_o(T)$. The path $\ell_o(T)$ will be called a \emph{branch} of $T$.

If we identify the Cantor set $C$ with the set of infinite binary words, then $(T_+, \sigma,T_-)$ corresponds to the homeomorphism of $C$ which for every leaf $o$ of $T_+$ maps each word $\ell_o(T_+)w$ to $\ell_{\sigma(o)}(T_-)w$ for every infinite binary word $w$. (It is easy to see that this is indeed a homeomorphism of $C$.) We say then that $\ell_o(T_+)\to \ell_{\sigma(o)}(T_-)$ is a \emph{branch} of the tree-diagram $(T_+,\sigma,T_-)$.

One can also identify a plane full binary tree $T$ with a decomposition of the unit interval $[0,1)$ where the endpoints of all subintervals in the decomposition are dyadic rationals.
Indeed, for a finite binary word $u$, we denote by $[u)$ the dyadic interval $[.u,.u1^{\mathbb N})$. Then if $T$ consists of the branches $u_1,\dots,u_n$, then $[u_1),\dots,[u_n)$ is a dyadic subdivision of the interval $[0,1)$. Then two full binary trees with the same number of leaves give two such decompositions with the same number of intervals, {and the element $g$ is represented by a map from $[0,1)$ to $[0,1)$ that maps each interval of the first decomposition linearly onto the corresponding (under $\sigma$) interval of the second decomposition.}

Two leaves $x,y$ of $T_+$ form a dipole if $x,y$ have a common parent in $T_+$ and $\sigma(x),\sigma(y)$ have a common parent in $T_-$. A dipole can be removed: we remove $x,y$ from the set of leaves of $T_+$ and $\sigma(x), \sigma(y)$ from the set of leaves of $T_-$, and re-identify the sets of leaves $L'$ of the two new trees $T'_+, T'_-$ in the natural way (from left to right). Then $\sigma$ gives us a permutation $\sigma'$ of the new set of leaves of the trees $T'_+, T'_-$. It is easy to see that the endomorphisms of the Cantor set $C$ corresponding to the  elements $(T_+, \sigma, T_-)$ and $(T'_+,\sigma',T'_-)$ are the same, so these two tree diagrams represent the same element of $V$, so these tree diagrams are \emph{equivalent}. In terms of branches, a dipole is a pair of branches $u0\to v0, u1\to v1$ and removing the dipole means that we replace that pair of branches by one branch $u\to v$. The inverse operation is called \emph{inserting a dipole}. A tree diagram without dipoles is called \emph{reduced} and every element of $V$ is represented by a unique reduced tree diagram.

For every  two elements $g,g'\in V$ represented by tree diagrams $(A_+,\alpha, A_-)$ and  \linebreak $(B_+, \beta, B_-)$,
one can insert several dipoles into these tree diagrams to obtain tree diagrams $(A'_+, \alpha', A'_-)$
and $(B'_+, \beta',B'_-)$ representing the same elements and such that $B'_+=A'_-$. Hence $\alpha'$ and  $\beta'$ are permutations from the same symmetric group. Then the reduced (after removing dipoles) form of $(A'_+,\alpha'\beta', B'_-)$ is the reduced tree diagram of the product $gg'$ in $V$. Here, homeomorphisms in $V$ and permutations in the symmetric group are composed from left to right. The inverse $g\iv$ of $g$ is represented by $(A_-,\alpha\iv, A_+)$. This describes the group structure of $V$. The group $T$ is a subgroup of $V$ which  consists of tree diagrams $(T_+, \sigma,T_-)$ where $\sigma$ is a {cyclic permutation} and $F$ is a subgroup of $T$ consisting of tree diagrams $(T_+,\sigma,T_-)$ where $\sigma$ is the identity  permutation. Thus $F$ acts by homeomorphisms on the unit interval $[0,1]$ and $T$ acts by homeomorphisms on the circle.

Most papers devoted to $F, T, V$ use the algebraic structure of these groups or dynamical properties of the action of $T$ on the circle $S^1$ \cite{GhS} or of $V$ on the Cantor set (see \cite{BBGGHMS, BS} and other papers) or representation of $V$ as (certain kinds of) diagram groups \cite{GubSa}. The known facts about geometry of $V$ and its subgroups $F, T$ are few. One can mention  the (highly non-trivial and surprising) result from \cite{GubaInv} that the Dehn function of $F$ is quadratic, the fact \cite{Farley} that $V$ acts properly by isometries on a locally finite dimensional CAT(0) cubical complex, and the result about distortion of various subgroups of $F, T, V$. For example, by \cite{BBGGHMS} all cyclic subgroups of $V$ are undistorted. The reason for only a few geometric results is that R. Thompson groups are clearly very far from being hyperbolic. Since these groups are simple or ``almost simple'' (as $F$), they cannot be, say, acylindrically hyperbolic, which is currently the largest well-studied class of groups which are ``close'' to being hyperbolic.

Still, the three geometric results mentioned above show that $F, T, V$ have some hyperbolicity (or at least some non-negative curvature) and further geometric properties of these groups should be studied.

One such property is existence of cut-points in asymptotic cones, or, equivalently \cite{DMS}, a superlinear divergence function. Divergence functions of metric spaces were introduced and studied in \cite{DMS}. If $G$ is a finitely generated group and $\Gamma$ is its Cayley graph corresponding to some finite generating set, then the divergence function of $G$ is the smallest function $f(n)$ such that every two vertices of $\Gamma$ at distance $n$ from the identity can be connected by a path in $\Gamma$ of length $\le f(n)$ and avoiding the ball of radius $\frac 14 n$ with a center at $1$. If we replace in the definition $\frac 14$  by any other constant $\delta<\frac 14$ we get an equivalent function (with the natural equivalence of functions used in Geometric Group Theory). Superlinear  divergence is equivalent to existence of cut points in some asymptotic cones of $G$. Since asymptotic cones of acylindrically hyperbolic groups are tree-graded spaces \cite{Sisto} (hence every point there is a cut-point), a superlinear divergence function shows some hyperbolicity of the group. On the other hand, a linear divergence function is of course also a geometric property of the group so both linear and superlinear divergence functions are interesting from the geometric point of view.

In \cite{DrSa, DMS} it is proved that the following finitely generated groups have linear divergence:

\begin{itemize}
	\item Direct products of infinite groups,
	\item Non-virtally cyclic groups satisfying a non-trivial law,
	\item Non-virtually cyclic groups which have infinite central cyclic subgroups,
	\item Uniform lattices in higher rank semi-simple Lie groups,
	\item Lattices in ${\mathbb Q}$-rank one semi-simple Lie groups of real rank $\ge 2$,
	\item $\SL_n(\Z), n\ge 3$.
\end{itemize}

It is natural then to ask (and the second author was asked several times by various people) whether the R. Thompson groups $F, T, V$ have linear divergence. In this paper, we give an affirmative answer to this question.

Our task was made easier by the fact that some convenient generating sets of $F$, $T$, $V$ are known \cite{CFP}. Namely, Thompson group $F$ is generated by the elements $x_0,x_1$ depicted in Figure \ref{fig:x0x1}. Thompson group $T$ is generated by $x_0,x_1,c_1$ where $c_1$ is depicted in Figure \ref{fig:c1}. Thompson group $V$ is generated by $x_0,x_1,c_1$ and $\pi_0$, where $\pi_0$ is depicted in Figure \ref{fig:pi0}.

\begin{figure}[ht]
	\centering
	\begin{subfigure}{.5\textwidth}
		\centering
		\includegraphics[width=.5\linewidth]{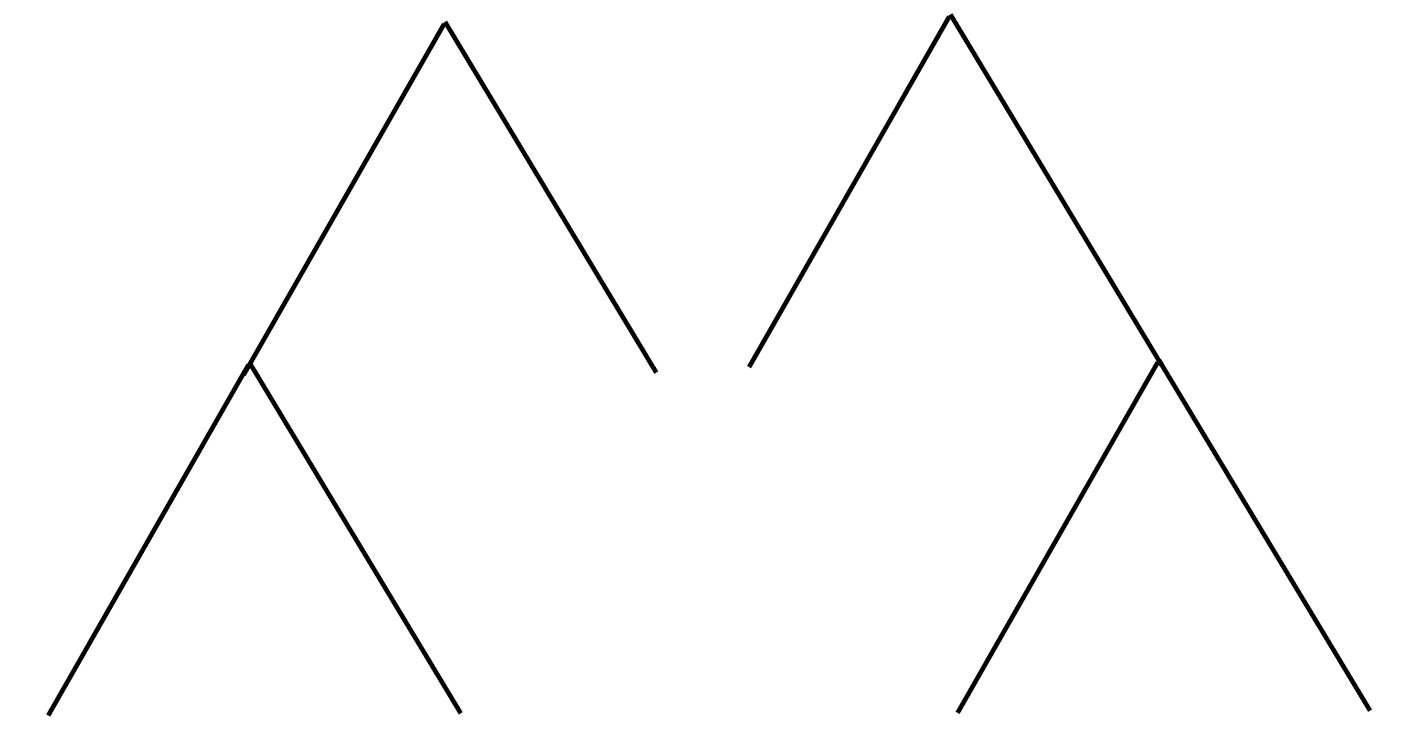}
		\caption{}
		\label{fig:x0}
	\end{subfigure}%
	\begin{subfigure}{.5\textwidth}
		\centering
		\includegraphics[width=.5\linewidth]{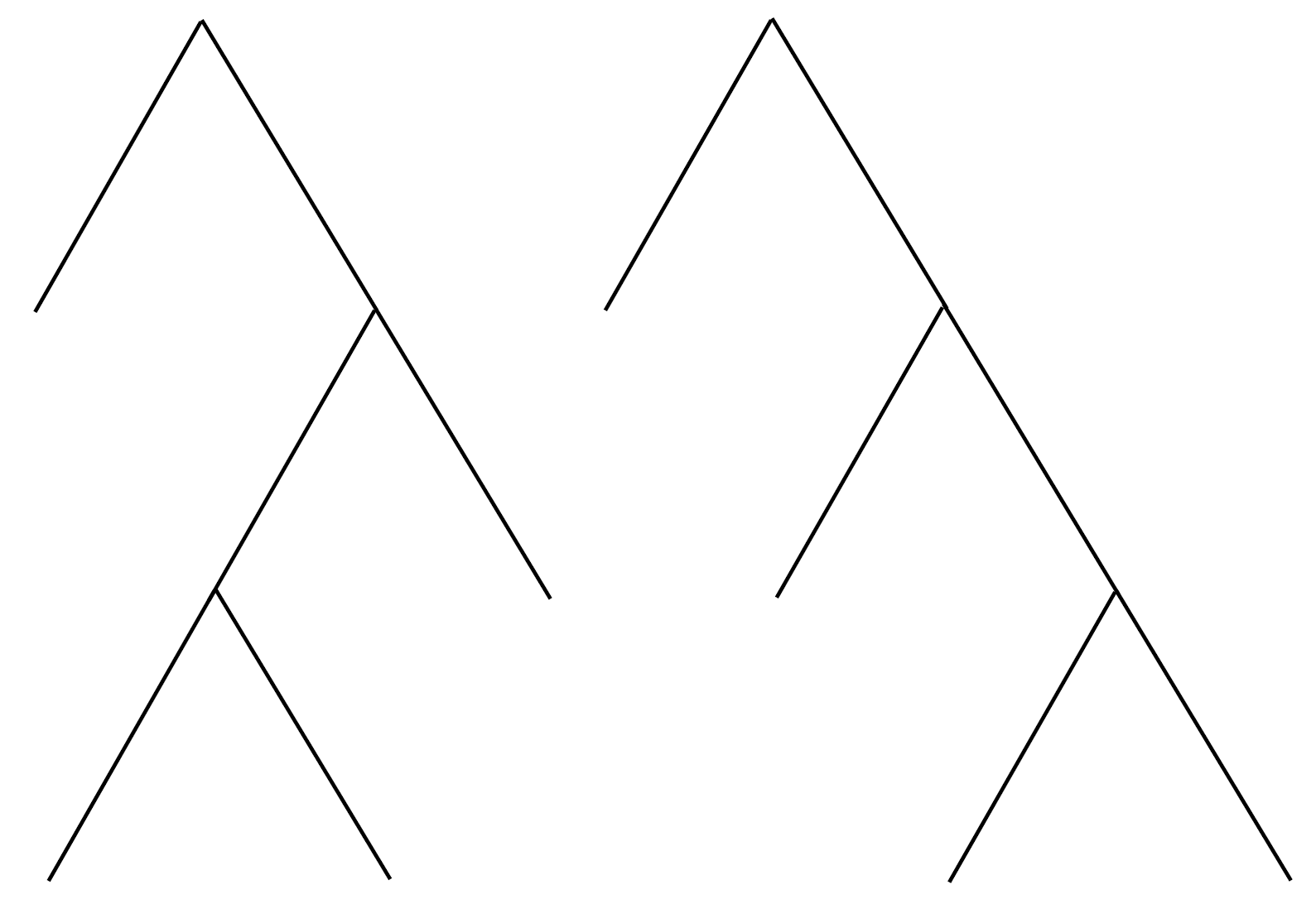}
		\caption{}
		\label{fig:x1}
	\end{subfigure}
	\caption{(a) The tree-diagram of $x_0$. (b) The tree-diagram of $x_1$. In both figures, $T_+$ is on the left, $T_-$ is on the right and the permutation $\sigma$ is the identity.}
	\label{fig:x0x1}
\end{figure}

\begin{figure}[ht]
	\centering
	\begin{subfigure}{.5\textwidth}
		\centering
		\includegraphics[width=.5\linewidth]{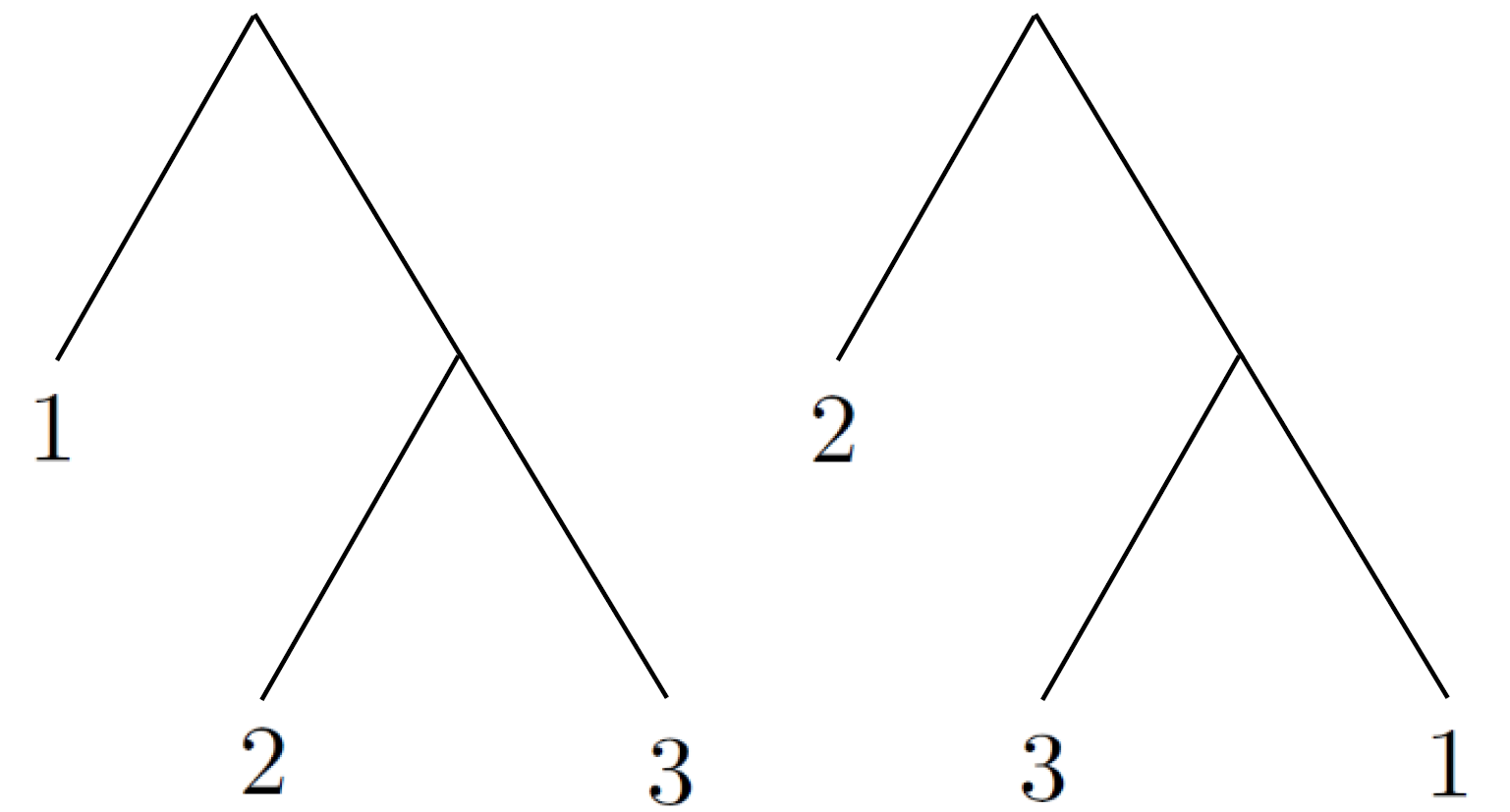}
		\caption{}
		\label{fig:c1}
	\end{subfigure}%
	\begin{subfigure}{.5\textwidth}
		\centering
		\includegraphics[width=.5\linewidth]{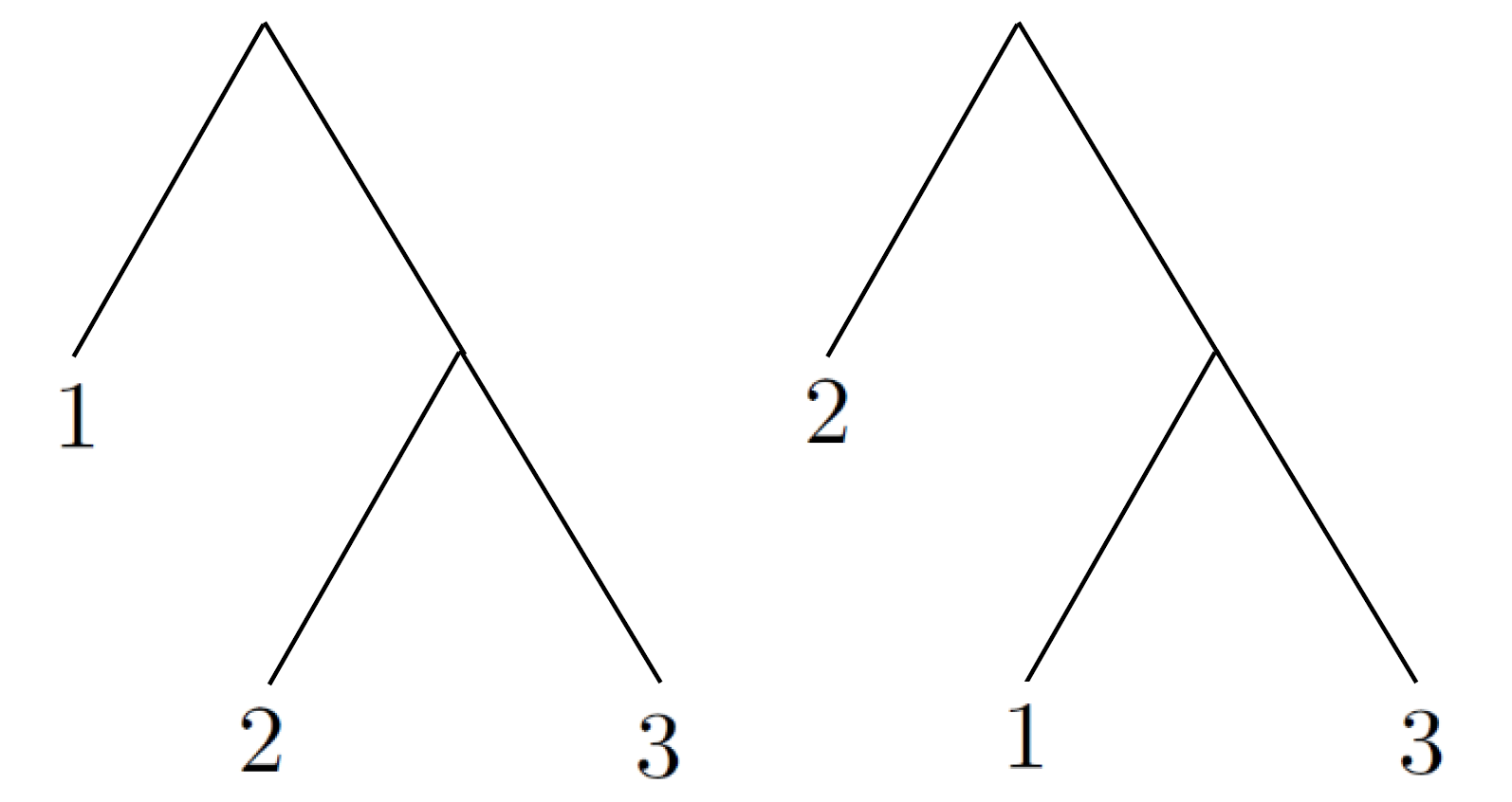}
		\caption{}
		\label{fig:pi0}
	\end{subfigure}
	\caption{(a) The tree-diagram of $c_1$. (b) The tree-diagram of $\pi_0$. In both figures, $T_+$ is on the left, $T_-$ is on the right and the permutation $\sigma$ maps each leaf of $T_+$ to the leaf of $T_-$ with the same label.}
	\label{fig:c1pi0}
\end{figure}

It was also helpful to have information about length functions of these groups. For example, in the case of the groups $\SL_n(\Z)$, Dru\c tu, Mozes and Sapir used the result of Lubotzky-Mozes-Raghunatan about the length function of that group \cite{LMR}. In the cases of $F$ and $T$ good enough information can be found in \cite{BCS} and \cite{BCST} (see Proposition \ref{Cc} below). But in the case of $V$ only partial information was known \cite{B}. Thus the case of the group $V$ turned out to be more complicated. However, after we completed the proof that $V$ has  linear divergence function, we realized that the same proof works for $F$ and $T$ as well.

The proof for $V$ works (very roughly) as follows. With every natural number $n$
we associate an element $g_n \in F$. The element $g_n$, viewed as a homeomorphism of $[0,1]$ is supported in a small subinterval $[0,a_n]$. Its word-length is approximately $Qn$ for some large constant $Q$. Given two natural numbers  $m,n$, it is easy to get from $g_m$ to $g_n$ via a path which avoids a ball around $1$ of radius $\delta\min\{m,n\}$ and has length linear in $m+n$. To finish, we prove that one can get from each element $h$ to the element $g_{|h|}\in F$ via a path which avoids a ball of radius $\delta|h|$ about the identity and has length linear in $|h|$. To construct that path, we find an element $s\in V$ such that $hs$ fixes the interval $[0,a_{|h|}]$ pointwise, so that $hs$ commutes with $g_{|h|}$. Then we move $h\to hs\to g_{|h|}hs\to g_{|h|}$.


It would be interesting to know whether that proof can be extended to other ``Thompson-like'' groups, such as higher dimensional Thompson groups \cite{Brin}, braided Thompson groups \cite{Brin1}, Lodha-Moore group \cite{LM} , topological full groups \cite{JM}, etc. In any case it would be interesting to know if these groups have linear divergence functions.

\section{The proof}

Let $\mathcal{A}=\{x_0,x_1\}$, $\mathcal{B}=\{x_0,x_1,c_1\}$ and $\mathcal{C}=\{x_0,x_1,c_1,\pi_0\}$ be the standard generating sets of $F$, $T$ and $V$ respectively. For an element $g\in V$ we denote by $\mathcal N(g)$ the number of leaves in each of the trees in the reduced tree diagram representing $g$.
We will use the following Proposition. Part (1) follows from \cite{BCS}. Part (2) follows from \cite{BCST}. Part (3) follows from \cite{B}.

\begin{Proposition}\label{Cc}
	There exist constants $0<c<1$ and $C>1$ such that the following hold.
	\begin{enumerate}
		\item[(1)] For every $g\in F$ we have $c\mathcal N(g)\le|g|_{\mathcal{A}}\le C\mathcal N(g).$
		\item[(2)] For every  $g\in T$  we have	$c\mathcal N(g)\le|g|_{\mathcal{B}}\le C\mathcal N(g).$
		\item[(3)] For every  $g\in V$ we have	$c\mathcal N(g)\le|g|_{\mathcal{C}}.$
	\end{enumerate}
\end{Proposition}

Let $g\in V$ with reduced tree-diagram $(T_+,\sigma,T_-)$. We will call $T_+$ the \emph{domain-tree} of $g$ and $T_-$ the \emph{range-tree} of $g$.  If other elements are considered as well, we will sometimes denote the domain-tree and range-tree of $g$ by $T_+(g)$ and $T_-(g)$, respectively. We will say that $u\rightarrow v$ is a \emph{branch} of $g$ if it is a branch of the reduced tree-diagram of $g$. If $u\rightarrow v$ is a branch of some tree-diagram of $g$ we say that $g$ \emph{takes the branch $u$ onto $v$}. Note that $g$ takes the branch $u$ onto $v$ if and only if $g$ maps the dyadic interval $[u)$ linearly onto $[v)$. Note also that if $g$ takes the branch $u$ onto $v$ then for some common (possibly empty) suffix $w$ of $u$ and $v$ we have $u\equiv u'w$ \footnote{$p\equiv q$ denotes letter-by-letter equality of words $p, q$.}, $v\equiv v'w$ and $g$ has the branch $u'\rightarrow v'$.

 Given a tree $T$, we denote by  $\ell_0(T)$ the length of the left most branch of $T$, that is, $\ell_0(T)=\ell$ if and only if $0^\ell$ is a branch of $T$. Similarly, we denote by $\ell_1(T)$ the length of the right most branch of $T$. Let $T$ be a non-empty binary tree. We note that $T_+(x_0)$ is a rooted subtree of $T$ if and only if $\ell_0(T)\neq 1$. Indeed, if $\ell_0(T)\neq 1$ then $00$ is a prefix of some branch of $T$ and as such, $01$ and $1$ are also prefixes of some branches of $T$. For an element $g\in V$ we let $\ell_i(g):=\ell_i(T_-(g))$, $i=1,2$.

 We will need the following lemmas.

\begin{Lemma}\label{x0}
	Let $g\in V$ with reduced tree diagram $(T_+,\sigma,T_-)$ and assume that $\mathcal N(g)\ge 4$. 
	Then
	$$\N(g)-1\le\N(gx_0)\le\N(g)+1.$$
	In addition, 	
	\begin{enumerate}
		\item[(1)] If $\ell_0(g)=1$ (i.e., if $T_+(x_0)$ is not a rooted subtree of $T_-$) then $\N(gx_0)=\N(g)+1$ and $\ell_0(gx_0)=1$.
		\item[(2)] If $\ell_0(g)\neq 1$ then $\N(gx_0)=\N(g)$ or
		$\N(gx_0)=\N(g)-1$. Moreover, $\ell_0(gx_0)=\ell_0(g)-1$.	
		\item[(3)] If $\ell_0(g)\neq 1$ and either $1$ or $01$ is a strict prefix of some branch of $T_-$ then $\N(gx_0)=\N(g)$ and $1$ is a strict prefix of some branch of $T_-(gx_0)$.
	\end{enumerate}
\end{Lemma}

\begin{proof}
	
	Assume first that $\ell_0(g)\neq 1$ so that $T_+(x_0)$ is a rooted subtree of $T_-(g)$, as in Figure \ref{fig:T-}. Then to multiply $g$ by $x_0$ we replace the standard tree-diagram of $x_0$ by an equivalent tree-diagram $(R_+,\mathrm{Id},R_-)$  where $R_+=T_-(g)$. We note that $R_-$ is obtained from $T_-(g)$ by removing its rooted subtree $T_+(x_0)$ and inserting the rooted subtree $T_-(x_0)$ instead (see Figure \ref{fig:x0equiv}). Then $(T_+(g),\sigma, R_-)$ is a tree-diagram of $gx_0$. Since $R_-$ was obtained from $T_-(g)$ by replacing its rooted subtree $T_+(x_0)$ by $T_-(x_0)$, the branches of the tree diagram $(T_+(g),\sigma, R_-)$ are obtained from those of $g=(T_+(g),\sigma,T_-(g))$ as follows: if $u\rightarrow wv$ is a branch of $g$ for  $w\equiv00$ or $w\equiv01$ or $w\equiv1$ then the tree-diagram $(T_+(g),\sigma, R_-)$ has a branch $u\rightarrow w'v$ for $w'\equiv0$ or $w'\equiv 10$ or $w'\equiv 11$, respectively. All branches of $(T_+(g),\sigma, R_-)$ are obtained in this way. In particular, $\ell_0(R_-)=\ell_0(T_-(g))-1=\ell_0(g)-1$.
	

	\begin{figure}[ht]
	\centering
	\includegraphics[width=.3\linewidth,height=.3\linewidth]{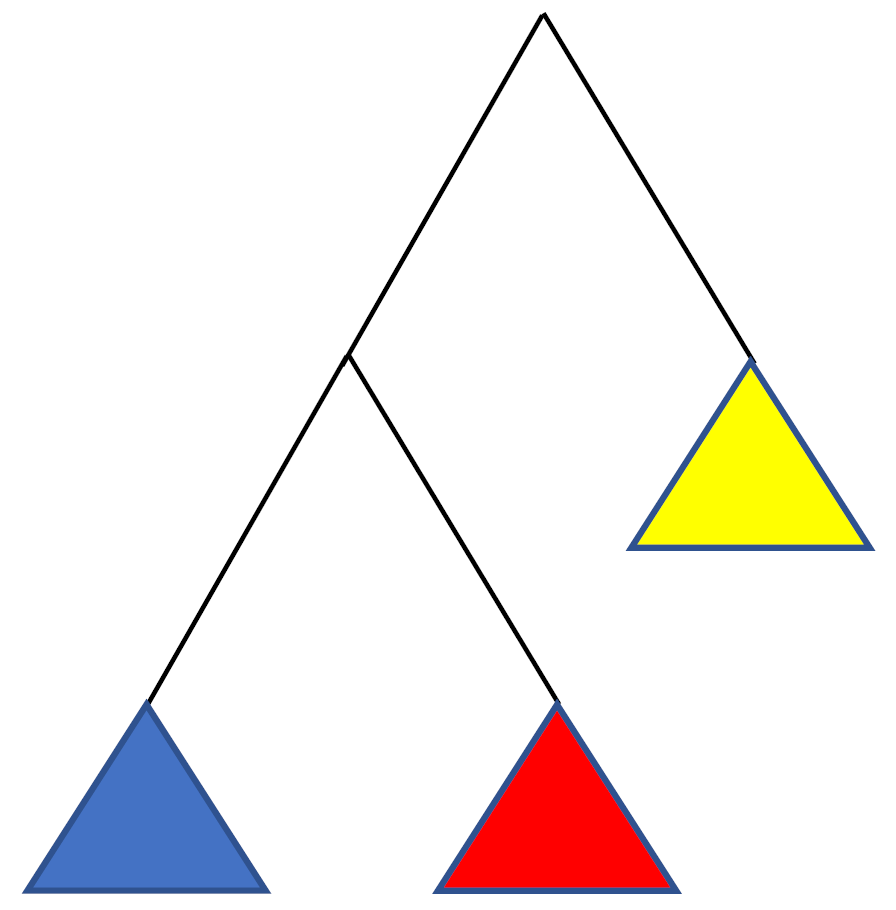}
	\caption{The tree $T_-(g)$ when $\ell_0(g)\neq 1$.}
	\label{fig:T-}
\end{figure}

	\begin{figure}[ht]
	\centering
	\includegraphics[width=.65\linewidth,height=.3\linewidth]{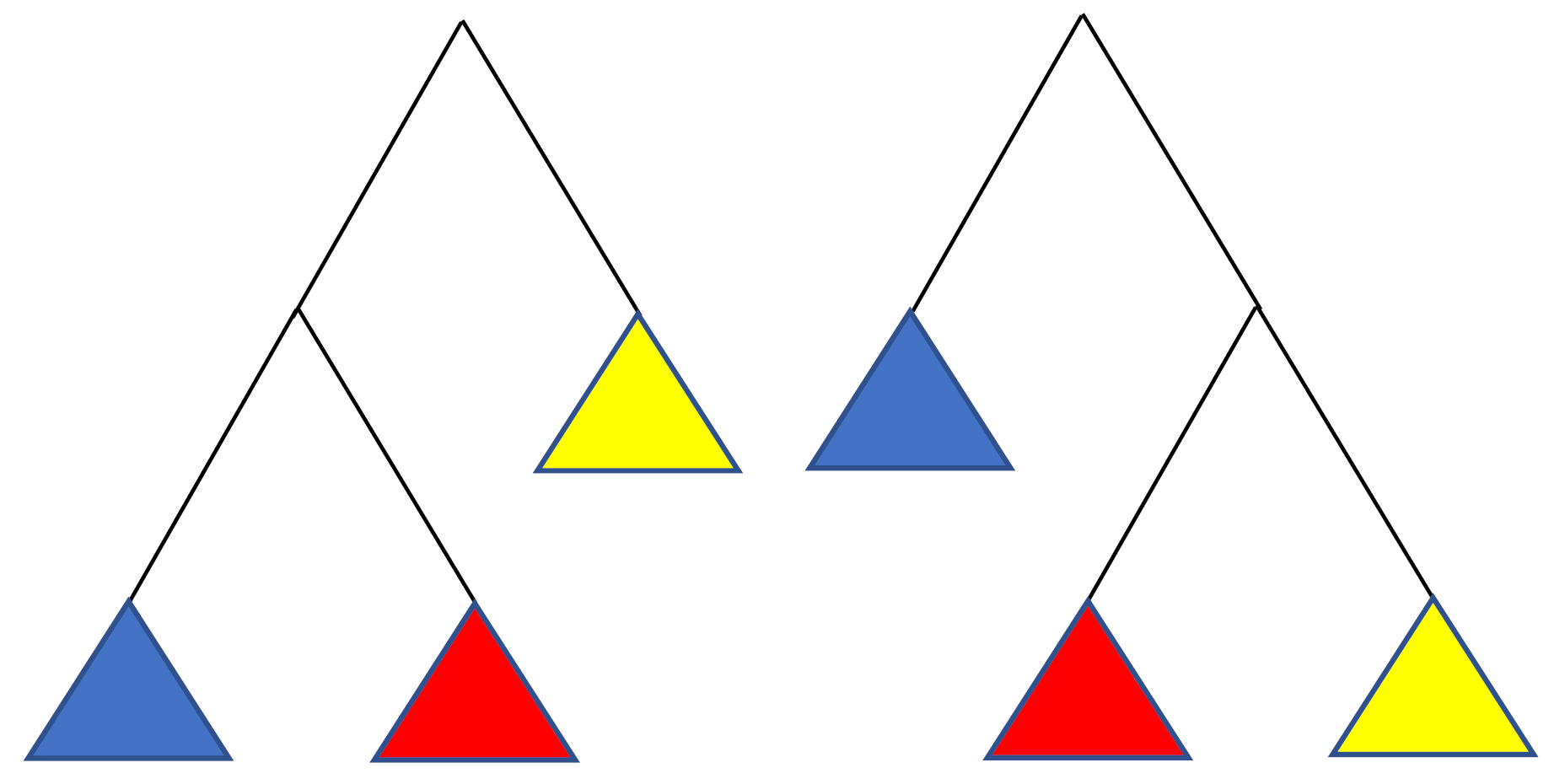}
	\caption{The tree-diagram $(R_+,\mathrm{Id},R_-)$ of $x_0$.}
	\label{fig:x0equiv}
\end{figure}

	Now $(T_+(g),\sigma, R_-)$ is reduced if and only if 
	it has no pair of branches of the form $u0\rightarrow t0$ and $u1\rightarrow t1$. If $(T_+(g),\sigma, R_-)$ is not reduced, let  $u0\rightarrow t0$ and $u1\rightarrow t1$ be such a pair of branches. We note that $t$ is not empty as $1$ is not a branch of $R_-$. If $0$ is a prefix of $t$, then $g=(T_+,\sigma,T_-)$ has the pair of branches $u0\rightarrow 0t0$ and $u1\rightarrow 0t1$ in contradiction to it being reduced. Hence, $1$ must be a prefix of $t$. If $10$ is a prefix of $t$ or $11$ is a prefix of $t$, the assumption that $(T_+,\sigma,T_-)$ is reduced yields a contradiction in a similar way. Hence $t\equiv 1$. So $R_-$ has branches of the form $10$ and $11$. Clearly, a single reduction of the common caret of $(T_+(g),\sigma, R_-)$ associated with these branches yields the reduced tree diagram of $gx_0$. Indeed, when reducing the common caret, we remain with the branch $1$ in $R_-$. As $\N(g)\ge 4$, the tree $R_-$ had at least $4$ leaves to begin with and so $0$ is a strict prefix of some branch of $R_-$. Hence, no further reductions can occur.
	
	It follows that, if $\ell_0(g)\neq 1$ then either $(T_+(g),\sigma, R_-)$ is reduced or the reduced-tree diagram of $gx_0$ is obtained from $(T_+(g),\sigma,R_-)$ by a single reduction which does not affect the left most branch of $R_-$. Hence $\N(gx_0)=\N(g)$ or $\N(gx_0)=\N(g)-1$ and $\ell_0(gx_0)=\ell_0(R_-)=\ell_0(g)-1$. Hence, part (2) holds.
	
	In the conditions of part (3) of the lemma the tree-diagram $(T_-(g),\sigma, R_-)$ is reduced. Indeed, since $1$ or $01$ is a strict prefix of some branch of $T_-(g)$, either $11$ or $10$ is a strict prefix of some branch of $R_-$. As noted above, in that case $(T_+(g),\sigma, R_-)$ is reduced. In addition, $1$ is a strict prefix of some branch of $T_-(gx_0)=R_-$. Hence, part (3) holds.
	
	Now assume that $\ell_0(g)=1$. To multiply $g$ by $x_0$, we replace the tree diagram $(T_+,\sigma,T_-)$ by an equivalent tree diagram $(T_+',\sigma',T_-')$ by attaching a common caret to the first leaf of $T_-$ (i.e., to the branch $0$ of $T_-$) and to the corresponding leaf of $T_+$. Then $T_+(x_0)$ is a rooted subtree of $T_-'$ and we can proceed as above. Let $R_-$ be the tree obtained from $T_-'$ by replacing the rooted subtree $T_+(x_0)$ with the rooted subtree $T_-(x_0)$. Then $(T_+',\sigma', R_-)$ is a tree-diagram of $gx_0$. To complete the proof it suffices to prove that $(T_+',\sigma', R_-)$ is reduced. Indeed, in that case, $\N(gx_0)=\N(g)+1$ and $\ell_0(gx_0)=\ell_0(R_-)=1$. By construction, the first two branches of $T_-'$ are $00$ and $01$. Hence the first two branches of $R_-$ are $0$ and $10$ and their leaves  do not form a dipole because they do not have a common parent in $R_-$. Every other branch of $R_-$ has $11$ as a strict prefix. Assume by contradiction that $u0\rightarrow 11t0$ and $u1\rightarrow 11t1$ are branches of  $(T_+',\sigma', R_-)$ forming a dipole. Then, from the construction of $R_-$, the tree $T_-$ has consecutive branches $1t0$ and $1t1$ and $g=(T_+,\sigma,T_-)$ has branches $u0\rightarrow 1t0$ and $u1\rightarrow 1t1$ in contradiction to it being reduced. Hence, part (1) holds. 	
\end{proof}

\begin{Corollary}\label{x0cor}
	Let $g\in V$ be such that $\mathcal N(g)\ge 4$. Let $\ell=\ell_0(g)$. Then for every $i\ge 0$ we have
	$$\N(gx_0^i)\ge \N(g)+i-2(\ell-1).$$
	Moreover, if either $1$ or $01$ is a strict prefix of some branch of $T_-(g)$ then
		$$\N(gx_0^i)= \max\{\N(g),\N(g)+i-(\ell-1)\}.$$
\end{Corollary}

\begin{proof}
	Assume first that $\ell=1$. Then applying Lemma \ref{x0}(1) iteratively shows that for every $i\ge 0$,
	\begin{equation}\label{eq1}
	\N(gx_0^i)=\N(g)+i.
	\end{equation}
	Thus, we can assume that $\ell>1$. By Lemma \ref{x0}(2), $\ell_0(gx_0^{\ell-1})=1$ and for every $i\le \ell-1$
	\begin{equation}\label{eq2}
	\N(gx_0^{i})\ge\N(g)-i\ge\N(g)-(\ell-1)\ge N(g)+i-2(\ell-1),
	\end{equation}
	where the last inequality follows from $i\le \ell-1$. Then applying the case $\ell=1$ (equation \ref{eq1}) to the element $gx_0^{\ell-1}$ we have that for all $j\ge 0$
	\begin{equation}\label{eq3}
	\N(gx_0^{\ell-1}x_0^j)=\N(gx_0^{\ell-1})+j\ge\N(g)-(\ell-1)+j.
	\end{equation}
	Substituting $j=i-(\ell-1)$ in inequality \ref{eq3}  gives that for every $i\ge \ell-1$,
	\begin{equation}\label{eq4}
	\N(gx_0^i)\ge\N(g)+i-2(\ell-1).
	\end{equation}
	It follows from inequalities \ref{eq2} and \ref{eq4} that for every $i\ge 0$,  	
	\begin{equation}\label{eq5}
	\N(gx_0^i)\ge\N(g)+i-2(\ell-1),
	\end{equation}
	as required.
	
	Now assume that  either $1$ or $01$ is a strict prefix of some branch of $T_-(g)$. Then Lemma \ref{x0}(2,3) implies that $\N(gx_0^i)= \N(g)$ for every $i\in\{0,\dots,\ell-1\}$. Thus, it suffices to prove that for every $i\ge \ell-1$ we have
	$\N(gx_0^i)= \N(g)+i-(\ell-1)$.
	But, as noted in inequality \ref{eq3}, for every $j\ge 0$ \begin{equation}
	\N(gx_0^{\ell-1}x_0^j)=\N(gx_0^{\ell-1})+j.
	\end{equation}
	Since $\N(gx_0^{\ell-1})= \N(g)$, we get that for every $j\ge 0$
	\begin{equation}\label{eq6}
	\N(gx_0^{\ell-1+j})=\N(g)+j.
	\end{equation}
	Then for every $i\ge \ell-1$ we have
	\begin{equation}\label{eq7}
	\N(gx_0^{i})=\N(g)+i-(\ell-1),
	\end{equation}
	as required.
\end{proof}

The proofs of the following lemma and corollary are symmetric to those of Lemma \ref{x0} and Corollary \ref{x0cor}.

\begin{Lemma}\label{x0inv}
	Let $g\in V$ with reduced tree diagram $(T_+,\sigma,T_-)$ and assume that $\N(g)\ge 4$.
	Then
	$$\N(g)-1\le\N(gx_0^{-1})\le\N(g)+1.$$
	In addition, 	
	\begin{enumerate}
		\item[(1)] If $\ell_1(g)=1$ (i.e., if $T_+(x_0^{-1})$ is not a rooted subtree of $T_-$) then $\N(gx_0^{-1})=\N(g)+1$ and $\ell_1(gx_0^{-1})=1$.
		\item[(2)] If $\ell_1(g)\neq 1$ then $\N(gx_0^{-1})=\N(g)$ or
		$\N(gx_0^{-1})=\N(g)-1$. Moreover, $\ell_1(gx_0^{-1})=\ell_1(g)-1$.	
		\item[(3)] If $\ell_1(g)\neq 1$ and either $0$ or $10$ is a strict prefix of some branch of $T_-$ then $\N(gx_0^{-1})=\N(g)$ and $0$ is a strict prefix of some branch of $T_-(gx_0^{-1})$.
	\end{enumerate}
\end{Lemma}

\begin{Corollary}\label{x0invcor}
	Let $g\in V$ be such that $\mathcal N(g)\ge 4$. Let $\ell=\ell_1(g)$. Then for every $i\ge 0$ we have
	$$\N(gx_0^{-i})\ge \N(g)+i-2(\ell-1).$$
	Moreover, if either $0$ or $10$ is a strict prefix of some branch of $T_-(g)$ then
	$$\N(gx_0^{-i})= \max\{\N(g),\N(g)+i-(\ell-1)\}.$$
\end{Corollary}

The next lemma describes the result of multiplying an element of $V$ on the right by an element from $F$ of a specific form.
Recall that the \emph{support} of an element $g\in V$ is, by definition,  the closure of the set $\{x\in[0,1):g(x)\neq x\}$ in the interval $[0,1]$. For a finite binary word $u$, we denote by $[u]$ the dyadic interval $[.u,.u1^{\mathbb N}]$. We denote by $F_{[u]}$ the subgroup of $F$ of all elements with support in the interval $[u]$. Let $g$ be an element of $F$ represented by a tree-diagram $(T_+,\mathrm{Id},T_-)$. We map $g$ to an element in $F_{[u]}$, denoted by $g_{[u]}$ and referred to as the \emph{copy of $g$ in $F_{[u]}$}. To construct the element $g_{[u]}$ we start with a minimal finite binary tree $T$ which contains the branch $u$. We take two copies of the tree $T$. To the first copy, we attach the tree $T_+$ at the end of the branch $u$. To the second copy we attach the tree $T_-$ at the end of the branch $u$. The resulting trees are denoted by $R_+$ and $R_-$, respectively. The element $g_{[u]}$ is the one represented by the tree-diagram $(R_+,\mathrm{Id}, R_-)$. Note that if $g$ consists of branches $v_i\to w_i, i=1,...,k,$ and $B$ is the set of branches of $T$ which are not equal to $u$, then $g_{[u]}$ consists of branches $uv_i\to uw_i, i=1,...,k$, and $p\to p, p\in B$.

For example, the copies of the generators $x_0,x_1$ of $F$ in $F_{[1]}$ are depicted in Figure \ref{fig:1x0x1}. Note that $(x_0)_{[1]}=x_1$. More generally, if $\{x_j, j\ge 0\}$ is the standard infinite generating set of $F$ where for $j\ge 1$, $x_{j+1}=x_0^{-j}x_1x_0^j$ (see \cite{CFP}), then for every $j\ge 0$ we have $x_{j+1}=(x_j)_{[1]}.$

\begin{figure}[ht]
	\centering
	\begin{subfigure}{.5\textwidth}
		\centering
		\includegraphics[width=.5\linewidth]{x1.png}
		\caption{The tree-diagram of $(x_0)_{[1]}$} 
		\label{fig:1x0}
	\end{subfigure}%
	\begin{subfigure}{.6\textwidth}
		\centering
		\includegraphics[width=.5\linewidth]{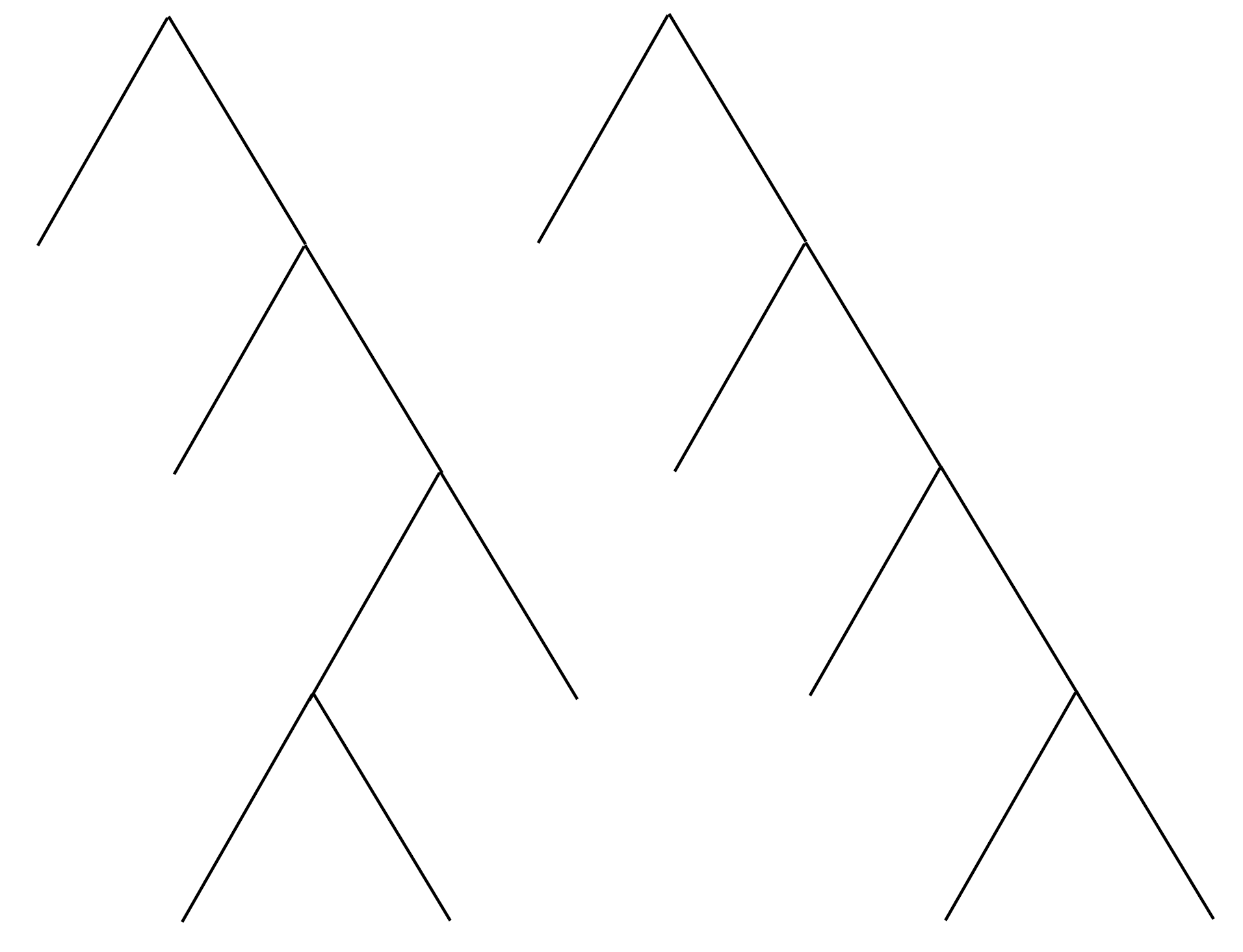}
		\caption{The tree-diagram of $(x_1)_{[1]}$}
		\label{fig:1x1}
	\end{subfigure}
	\caption{}
	\label{fig:1x0x1}
\end{figure}

\begin{Lemma}\label{bothsides}
	Let $g\in V$ be an element with reduced tree-diagram $(T_+,\sigma,T_-)$. Let $u\rightarrow v$ be a  branch of $g$ and let $h$ be an element of $F$. Let $h'=(h)_{[v]}$. Then
	$$\N(gh')=\N(g)+\N(h)-1.$$
	Moreover, if $h$ consists of branches $w_i\to z_i, i=1,...,k,$ and $B$ is the set of branches of $g$ which are not equal to $u\rightarrow v$, then $gh'$ consists of branches $uw_i\to vz_i, i=1,...,k$, along with all branches in $B$.
\end{Lemma}

\begin{proof}
	To multiply $g$ by $h'$ we note that the minimal refinement of $T_-(g)$ and $T_+(h')$ is the tree obtained from $T_-(g)$ by attaching the tree $T_+(h)$ at the bottom of the branch $v$. We denote by $S$ the described tree. Let $(R_1,\sigma',S)$ be an equivalent tree-diagram of $g$. Then $R_1$ is obtained from $T_+(g)$ by attaching a copy of $T_+(h)$ to the bottom of the branch $u$ (indeed, $u\rightarrow v$ is a branch of $g$). $\sigma'$ is obtained from $\sigma$ in the standard way. If $(S,\mathrm{Id},R_2)$ is a tree-diagram of $h'$ then $R_2$ can be obtained from $T_-(g)$ by attaching $T_-(h)$ at the bottom of the branch $v$. The product of the tree-diagrams $(R_1,\sigma',S)$ and $(S,\mathrm{Id},R_2)$ is $(R_1,\sigma',R_2)$. Note that this tree-diagram is obtained from $(T_+(g),\sigma,T_-(g))$ by attaching $T_+(h)$ to the bottom of the branch $u$ of $T_+(g)$, attaching $T_-(h)$ to the bottom of the branch $v$ of $T_-(g)$ and modifying the permutation $\sigma$ accordingly (so that the resulting branches are as described in the lemma). We note that this tree-diagram is reduced. Indeed, if $x$ and $y$ are consecutive leaves of the tree $R_1$, then either $x$ and $y$ are leaves of the subtree $T_+(g)$ of $R_1$, in which case they cannot form a dipole since $(T_+(g),\sigma,T_-(g))$ is reduced, or $x$ and $y$ are leaves of the copy of $T_+(h)$ and cannot form a dipole for similar reasons.

	The formula regarding the number of leaves of $gh'$ follows from the description of its reduced tree diagram. Indeed, the branches in that diagram are exactly those of $g$ other than the branch $u\rightarrow v$ which is replaced by $\N(h)$  branches.
\end{proof}

We provide a unified proof for the linear divergence of $F$, $T$ and $V$.
Below, $\mathcal G$ can stand for either one of these groups and $X$ denotes the standard generating set of $\mathcal G$.
If $\mathcal G=T$ then all permutations $\sigma$ in the tree diagrams below are cyclic permutations. If $\mathcal G=F$ then the permutations are all trivial.
For an element $g\in\mathcal G$ we denote by $|g|$ the word length $|g|_X$. If $w$ is a word over the alphabet $X$, we denote by $||w||$ the length of $w$. Note that any word $w$ over the alphabet $X$ can be viewed as an element of $\mathcal G$, such that $|w|\le ||w||$.

\begin{Proposition}\label{mainprop}
	There exist constants $\delta ,D>0$ and a positive integer $Q$ such that the following holds. Let $g\in\mathcal G$ be an element with $\N(g)\ge 4$. Then there is a path of length at most $D|g|$ in the Cayley graph $\Gamma=\Cay(\mathcal G,X)$ which avoids a $\delta|g|$-neighborhood of the identity and which has initial vertex $g$ and terminal vertex $x_0^{Q|g|}x_1^{-1}x_0^{-Q|g|+1}$.
	
	In other words, there is a word $w$ in the alphabet $X$ such that $||w||<D|g|$; for any prefix $w'$ of $w$, we have $|gw'|>\delta|g|$ and such that $$gw=x_0^{Q|g|}x_1^{-1}x_0^{-Q|g|+1}.$$
\end{Proposition}

\begin{proof}
	Let $C$ and $c$ be the constants from Proposition \ref{Cc}.	The word $w$ will be composed of $5$ subwords $w_1,\dots,w_5$ so that $w\equiv w_1\dots w_5$, that is we subdivide the path $w$ in the Cayley graph of $\mathcal G$ into 5 parts so that \begin{itemize} 
\item $g_1=gw_1$ does not have $0$ as a branch of $T_-(g_1)$, 
\item $g_2=g_1w_2$ satisfies $\N(g_2)>M\N(g_1)$ for large enough $M$,
\item $g_3=g_2w_3$ is such that $\ell_0(g_3)$ is not very large and  $0^{\ell_0(g_3)}\rightarrow 0^{\ell_0(g_3)}$ is a branch of $g_3$ (that is $g_3$ fixes pointwise a not too small interval $[0,\epsilon]$),
\item $g_4=g_3w_4$, the support of $w_4$ is in $[0^{\ell_0(g_3)}]$ so that $w_4$ commutes with $g_3$, $\N(g_4)$ is much bigger than $\N(g_3)$. Moreover $w_4\equiv x_0^{Q|g|}x_1^{-1}x_0^{-Q|g|+1}$ for a large enough $Q$.
\item $g_5=g_4w_5$ where $w_5=g_3\iv$. Then $g_5=g_4g_3^{-1}=w_4$ as desired. 
\end{itemize}	

We also make sure that none of the 5 subpaths of our path $w$ gets too close to the identity and that the length of each of them is bounded by a linear function in $|g|$.  
	\begin{Subword}\label{step1}
		If $0$ is not a branch of $T_-(g)$ we let $w_1\equiv \emptyset$ and let $g_1=g$.
		Otherwise, we let $w_1\equiv x_0^2x_1^{-1}x_0^{-1}$ and let $g_1=gw_1$.
	\end{Subword}
	\begin{Lemma}\label{lem1}
		$0$ is not a branch of $T_-(g_1)$. Moreover, $\N(g_1)\le\N(g)+2$ and for every prefix $w'$ of $w_1$ we have $\N(gw')\ge\N(g)$.
	\end{Lemma}
	
	\begin{proof}
		If $0$ is not a branch of $T_-(g)$ then the lemma is clear. Hence, we can assume that $0$ is a branch of $T_-(g)$.  Let $u$ be the corresponding branch of $T_+(g)$ so that $g$ has the branch $u\rightarrow 0$. We note that  $w_1=x_0^2x_1^{-1}x_0^{-1}$ is the $F_{[0]}$-copy of $x_0$.
		Hence, by Lemma \ref{bothsides}, for each branch $v_1\rightarrow v_2$ of $x_0$, we have that $uv_1\rightarrow 0v_2$ is a branch of $g_1=gw_1$. As such, $0$ is a strict prefix of some branch of $T_-(g_1)$ and is not a branch itself.
		
		For the second claim, we note that since $0$ is a branch of $T_-(g)$, we have $\ell_0(g)=1$. Hence, by Lemma \ref{x0}(1), $\N(gx_0)=\N(g)+1$ and $\ell_0(gx_0)=1$. Then another application of the lemma shows that $\N(gx_0^2)=\N(g)+2$. By Lemma \ref{bothsides}, $\N(g_1)=\N(gw_1)=\N(g)+2$ (recall that the number of leaves of $x_0$ is $3$). Since multiplying an element by $x_0$ can decrease the number of leaves of the element at most by $1$ (Lemma \ref{x0}), $\N(gx_0^2x_1)=\N(gw_1x_0)\ge\N(g)+1$.
	\end{proof}
	
	\begin{Subword}
		We fix an integer $M\ge100C/c$. Then we let $$w_2\equiv x_0^{-M\N(g_1)}x_1x_0^{M\N(g_1)}$$
		and $g_2=g_1w_2$.
	\end{Subword}
	
	\begin{Lemma}\label{lem2} The following assertions hold.
		\begin{enumerate}
			\item[(1)] For every prefix $w'$ of $w_2$, we have $\N(g_1w')\ge\N(g_1)$.
			\item[(2)] $\N(g_2)\ge M\N(g_1)$.
			
		\end{enumerate}
	\end{Lemma}
	
	\begin{proof}
		We start by proving part (2). First, we note that as an element of $\mathcal G$, we have
		$w_2=x_{m}$ where $m=1+M\N(g_1)$.
		Let $\ell_1=\ell_1(g_1)$ so that $1^{\ell_1}$ is a branch of $T_-(g_1)$ and there is some finite binary word $u$ such that $u\rightarrow 1^{\ell_1}$ is a branch of $g_1$. Clearly, $\ell_1<\N(g_1)$.
		We note that $x_m$ is the $F_{[1^{\ell_1}]}$ copy of $x_{m-\ell_1}$. Hence, by Lemma \ref{bothsides},  we have
		$$\N(g_2)=\N(g_1w_2)=\N(g_1x_m)=\N(g_1)+\N(x_{m-\ell_1})-1.$$
		Recall that for all $r$, $\N(x_r)=r+3$. Hence,
		$$\N(g_2)=\N(g_1)+m-\ell_1+2
		=\N(g_1)+M\N(g_1)+1-\ell_1+2\ge M\N(g_1),$$
		as $\ell_1<\N(g_1)$. Hence, part (2) holds.
		
		Before proving part (1), we note that if $z_j\rightarrow q_j$, $j=1,\dots,n$ are the branches of $x_{m-\ell_1}$ then by Lemma \ref{bothsides} the branches of $g_2=g_1x_m$ are $uz_j\rightarrow 1^{\ell_1}q_j$, $j=1,\dots,n$ as well as  all branches $a_k\rightarrow b_k$ which are branches of $g_1$, other than $u\rightarrow 1^{\ell_1}$.
		
		Now, let $w'$ be a prefix of $w_2$. Then either (a) $w'\equiv x_0^{-i}$ for some  $0\le i<m$, or (b) $w'\equiv x_0^{-m+1}x_1x_0^i$ for some $0\le i<m$, where $m$ is as above. We consider the element $g_1w'$ in each of these cases.
		
		By Lemma \ref{lem1}, $0$ is a strict prefix of some branch of $T_-(g_1)$. Hence, by Corollary \ref{x0invcor}, for any $i$,  $\N(g_1x_0^{-i})\ge\N(g_1)$. Hence, the lemma holds for subwords $w'$ of type (a). Now, consider the element $g_1w'$ for $w'$ of type (b). As an element of $\mathcal G$, we have $$g_1w'=g_1x_0^{-m+1}x_1x_0^i=g_1x_{m}x_0^{-m+1+i}.$$
		The proof of part (2) shows that $\N(g_1x_m)=\N(g_2)\ge\N(g_1)$. Since $i<m$, it suffices to show that multiplying $g_1x_m$ on the right by any negative power of $x_0$ cannot decrease the number of leaves. But, it is clear that in the range-tree of $g_1x_m$, $0$ is a strict prefix of some branch. Indeed, all branches of $T_-(g_1)$, other than $1^{\ell_1}$, are also branches of the range-tree of $g_1x_m$. Since $0$ is a strict prefix of some branch of $T_-(g_1)$, this is also true for the range-tree of $g_1x_m$. Hence, by Corollary \ref{x0invcor}, multiplying $g_1x_m$ on the right by any negative power of $x_0$ would result in a tree-diagram with at least as many leaves as in $g_1$.
	\end{proof}

	\begin{Subword}
		Let $(T_+,\sigma,T_-)$ be the reduced tree-diagram of $g_1$ and assume that $\sigma(1)=k$, where $1\le k\le\N(g_1)$.
		Let $h$ be the element of $T$ given by the tree-diagram $(T_-,\sigma',T_+)$, where $\sigma'$ is the cyclic permutation of $\{1, 2, \dots,\N(g_1)\}$ such that $\sigma'(k)=1$. We note that if $\mathcal G=F$ then $k=1$ and $h\in F$.
		Now, we let $w_3$ be the minimal word over $X$ such that $h=w_3$ in $\mathcal G$ and let $g_3=g_2w_3$. 	
	\end{Subword}

	\begin{Lemma}\label{lem3}
		The following assertions hold.
		\begin{enumerate}
			\item[(1)] $||w_3||\le C\N(g_1)$.
			\item[(2)] $\N(g_3)\ge (M-1)\N(g_1)$.
			\item[(3)] $\ell_0(g_3)\le\N(g_1)$ and  $0^{\ell_0(g_3)}\rightarrow 0^{\ell_0(g_3)}$ is a branch of $g_3$.
		\end{enumerate}
	\end{Lemma}
	
	\begin{proof}
		(1) The reduced tree-diagram of $h$ has at most $\N(g_1)$ leaves. Since $h\in F$ if $\mathcal G=F$ and $h\in T$ in the general case, by Proposition \ref{Cc}, the word length of $h$ over $X$ is at most $C\N(g_1)$.
		
		For parts (2) and (3) we recall the structure of $g_2$ as described in the proof of Lemma \ref{lem2}. Let $\ell_1=\ell_1(g_1)$ and let $m=1+M\N(g_1)$. Then $g_2=g_1x_m=g_1(x_{m-\ell_1})_{[1^{\ell_1}]}$.
		Now, let $u$ be such that $u\rightarrow 1^{\ell_1}$ is a branch of $g_1$. Then if $z_j\rightarrow q_j$, $j=1,\dots,n$ are the branches of $x_{m-\ell_1}$ then the branches of $g_2$ are $uz_j\rightarrow 1^{\ell_1}q_j$, $j=1,\dots,n$ as well as all the  branches $a_k\rightarrow b_k$ of $g_1$, other than $u\rightarrow 1^{\ell_1}$.
		
		Since $1^{\ell_1}$ is a branch of $T_-$, there is some branch $v$ of $T_+$ such that $1^{\ell_1}\rightarrow v$ is a branch of the tree-diagram $(T_-,\sigma',T_+)$ representing $h$. Hence, for every $j=1,\dots,n$, the product $g_2h$ takes the branch $uz_j$ onto $vq_j$. Indeed, $g_2$ takes the branch $uz_j$ onto $1^{\ell_1}q_j$ and $h$ takes the branch $1^{\ell_1}q_j$ onto $vq_j$. We claim that $uz_j\rightarrow vq_j$, $j=1,\dots,n$ are all branches of the reduced tree-diagram of $g_2h$. Indeed, they are all branches of some tree-diagram of $h$. Since $z_j\rightarrow q_j$ are the branches of the reduced tree-diagram of $x_{m-\ell_1}$, no reductions can occur in the carets formed by the branches $uz_j\rightarrow vq_j$ of the tree-diagram of $g_2h$.
		It follows that $$\N(g_3)=\N(g_2h)\ge\N(x_{m-\ell_1})=m-\ell_1+3=M\N(g_1)+1-\ell_1+3\ge (M-1)\N(g_1),$$ since $\ell_1=\ell_1(g_1)<\N(g_1)$. This proves part (2).
		
		For part (3), let $r$ be such that $0^r$ is the first branch of $T_+$ and note that $r<\N(g_1)$. Let $s$ be the $k^{th}$ branch of $T_-$, so that $0^r\rightarrow s$ is a branch of $g_1$ (recall that $\sigma(1)=k$). By definition of $\sigma'$, we have that $s\rightarrow 0^r$ is a branch of the tree-diagram $(T_-,\sigma',T_+)$ representing $h$. Recall that $u\rightarrow 1^{\ell_1}$ is a branch of $g_1$. We consider two cases: (a) $u\not\equiv 0^r$, and (b) $u\equiv 0^r$.
		
		In case (a), $0^r\rightarrow s$ is a branch of $g_2$. Indeed, every branch of $g_1$, other than $u\rightarrow 1^{\ell_1}$, is also a branch of $g_2$. Then since $g_2$ takes the branch $0^r$ to $s$ and $h$ takes the branch $s$ to $0^r$, the product $g_2h$ takes the branch $0^r$ to $0^r$. In particular, for some $r'\le r$, the element $g_3=g_2h$ has the branch $0^{r'}\rightarrow 0^{r'}$. Then $\ell_0(g_3)=r'<\N(g_1)$, as required.
		
		In case (b), $u\equiv 0^r$, so $s\equiv 1^{\ell_1}$. Since $(T_-,\sigma',T_+)$ has the branch $s\equiv 1^{\ell_1}\rightarrow v$ and $s\rightarrow 0^r$, we get that $v\equiv 0^r\equiv u$. 
		As noted above, $g_3$ has the branch $uz_j\equiv0^rz_j\rightarrow vq_j\equiv 0^rq_j$ for every branch $z_j\rightarrow q_j$ of $x_{m-\ell_1}$, and in particular, for the branch $z_1\equiv 0\rightarrow q_1\equiv 0$ of $x_{m-\ell_1}$. Hence, $0^{r+1}\rightarrow0^{r+1}$ is a branch of (the reduced tree-diagram of) $g_3$. Then $\ell_0(g_3)=r+1\le\N(g_1)$, as required.
	\end{proof}

	\begin{Subword}
		We fix an integer $Q\ge 10M/c^2$ and let
		$$w_4\equiv x_0^{Q|g|}x_1^{-1}x_0^{-Q|g|+1}.$$
		We also let $g_4=g_3w_4$.
	\end{Subword}
	
	\begin{Lemma}\label{lem4}
		The following assertions hold.
		\begin{enumerate}
			\item[(1)] For every prefix $w'$ of $w_4$ we have
			$$\N(g_3w')\ge \N(g_3)+\frac{1}{2}||w'||-2\N(g_1).$$
			\item[(2)] As elements in $\mathcal G$, $g_3$ and $w_4$ commute.
		\end{enumerate}
	\end{Lemma}
	
	\begin{proof}
		To prove part (1), let $\ell=\ell_0(g_3)$. By Lemma \ref{lem3}, $\ell\le \N(g_1)$. We first consider prefixes $w'$ of $w_4$ which are positive powers of $x_0$. For each prefix $w'\equiv x_0^i$ for $i\ge 0$ we have by Corollary \ref{x0cor},
		\begin{equation}\label{*}
		\begin{split}
				\N(g_3w')=\N(g_3x_0^i)&\ge \N(g_3)+i-2(\ell-1)\\
				&\ge\N(g_3)+i-2(\N(g_1)-1)\\
				&\ge\N(g_3)+||w'||-2(\N(g_1)-1).
		\end{split}
		\end{equation}
		Thus, to finish the proof of part (1) it suffices to show that for every prefix $w'$ of $w_4$ which contains the letter $x_1^{-1}$, we have
		\begin{equation}\label{**}
		\N(g_3w')\ge\N(g_3x_0^{Q|g|})-1.
		\end{equation}
		Indeed, in that case,  by inequality \ref{*},
		\begin{equation}
		\begin{split}
		\N(g_3w')&\ge \N(g_3)+Q|g|-2(\N(g_1)-1)-1\\
		&=N(g_3)+\frac{1}{2}||w_4||-2\N(g_1)+1\\
		&\ge N(g_3)+\frac{1}{2}||w'||-2\N(g_1),
		\end{split}
		\end{equation}
		as necessary.
		To prove that inequality \ref{**} holds 
		we first consider the prefix
		$$p\equiv x_0^{Q|g|}x_1^{-1}x_0^{-1}\equiv x_0^{Q|g|-2}\cdot x_0^2x_1x_0^{-1}.$$
		Note that by Lemmas \ref{lem3}, \ref{lem1} and Proposition \ref{Cc},
		\begin{equation}\label{Qineq}
		\ell=\ell_0(g_3)\le\N(g_1)\le\N(g)+2\le \frac{1}{c}|g|+2<Q|g|-2.
		\end{equation}
		Hence, by Lemma \ref{x0}(1,2), $\ell_0(g_3x_0^{Q|g|-2})=1$ so that $0$ is a branch of $T_-(g_3x_0^{Q|g|-2})$. Recall also that $x_0^2x_1^{-1}x_0^{-1}$ is the $F_{[0]}$-copy of $x_0$. Hence, by Lemma \ref{bothsides}, we have
		\begin{align*}
		\N(g_3p)&=\N(g_3x_0^{Q|g|-2}\cdot x_0^2x_1^{-1}x_0^{-1})\\
		&=\N(g_3x_0^{Q|g|-2})+\N(x_0)-1\\
		&=\N(g_3x_0^{Q|g|})-2+3-1\\
		&=\N(g_3x_0^{Q|g|}),
		\end{align*}
		so inequality \ref{**} holds for the prefix $p$. We note also that by Lemma \ref{bothsides}, the range-tree of $g_3p$ is obtained from that of $g_3x_0^{Q|g|-2}$ by attaching a copy of $T_-(x_0)$ to the bottom of the branch $0$. In particular, $0$ is a strict prefix of some branch of $T_-(g_3p)$.
		
		To finish the proof, it remains to prove that inequality \ref{**} holds for the prefix (a) $q\equiv x_0^{Q|g|}x_1^{-1}$ and prefixes of the form (b) $w'\equiv x_0^{Q|g|}x_1^{-1}x_0^{-i}$, for $i>1$.
		For the case (a), we note that $g_3q=g_3px_0$. Hence, by Lemma \ref{x0}(2),
		$$\N(g_3q)\ge\N(g_3p)-1\ge\N(g_3x_0^{Q|g|})-1,$$
		as required. Finally, prefixes of the form (b) can be written as $w'\equiv px_0^{-(i-1)}$. As noted above, $0$ is a strict prefix of some branch of $T_-(g_3p)$. Hence, by Corollary \ref{x0invcor} we have
		$$\N(g_3w')=\N(g_3px_0^{-(i-1)})\ge\N(g_3p)=\N(g_3x_0^{Q|g|}),$$
		as required.
		
		For part (2), we note that by Lemma \ref{lem3}, $g_3$ has the branch $0^\ell\rightarrow 0^\ell$. As such, it fixes the interval $[0^\ell)$ pointwise. Thus, to prove that $w_4$ commutes with $g_3$ it suffices to prove that the support of $w_4$ is contained in the interval $[0^{\ell}]$. Let $f=x_0^2x_1^{-1}x_0^{-1}$ and recall that $f$ is the $F_{[0]}$-copy of $x_0$. In particular, the support of $f$ is contained in the interval $[0]$. Since $x_0^{-1}$ has the branch $00\rightarrow 0$, any conjugate $f^{x_0^{-i}}$ for $i\ge 0$ has support in the interval $[0^{1+i}]$. In particular, $w_4=f^{x_0^{-(Q|g|-2)}}$ has support in the interval $[0^{Q|g|-1}]$. By inequality \ref{Qineq}, $\ell<Q|g|-1$. Hence, $w_4$ has support in the interval $[0^\ell]$, as required.

	\end{proof}
	
	\begin{Subword}
		Let $w_5$ be a minimal word in the alphabet $X$ such that $w_5=g_3^{-1}$ in $\mathcal G$. Let $g_5=g_4w_5$.
	\end{Subword}
	
	It follows from Lemma \ref{lem4}(2) that
	$$g_5=g_3w_4g_3^{-1}=w_4=x_0^{Q|g|}x_1^{-1}x_0^{-Q|g|+1}.$$
	Hence, for $w\equiv w_1w_2w_3w_4w_5$ we have that $gw=x_0^{Q|g|}x_1^{-1}x_0^{-Q|g|+1}$, as required.
	
	It remains to prove that one can choose constants $\delta,D$ (independently of $g$), so that the path $w$ satisfies the conditions in the proposition.
	For that, we note that $||w_5||=|g_3|\le|g|+||w_1w_2w_3||$. Upper bounds for the lengths of $w_1,w_2,w_3,w_4$ follow from their definition. Together we have the following.
	\begin{align*}
	||w||&\le 2||w_1w_2w_3||+||w_4||+|g|\\
	&\le 2(4+2M\N(g_1)+1+C\N(g_1))+2Q|g|+|g|\\
	&\le 8M\N(g)+3Q|g|\\
	&\le 8M/c|g|+3Q|g|.
	\end{align*}
	Where the third inequality follows from $\mathcal N(g_1)\le \mathcal N(g)+2$ and from the definition of $M$ and $Q$. The last inequality follows from Proposition \ref{Cc}, since $\mathcal N(g)\le\frac{1}{c}|g|$. Hence, fixing $D=8M/c+3Q$ ensures that $||w||\le D|g|$, as required. Moreover, we note that, as in the above calculation, $||w_1w_2w_3||\le 4M\mathcal N(g)$.
	Now, let $\delta=\frac{c}{8M}$. The following  lemma completes the proof of the proposition.
	
	\begin{Lemma}
		Let $w'$ be a prefix of $w$. Then $|gw'|\ge \delta|g|.$
	\end{Lemma}
	\begin{proof}
		We separate the proof into two cases depending on the length of $g$.
		
		Case (1): $|g|<8M\N(g)$.
		
		It follows from Lemmas \ref{lem1} and \ref{lem2} that for every prefix $w'\le w_1w_2$, we have $\mathcal N(gw')\ge\mathcal N(g)$.
		Then, applying Proposition \ref{Cc}, we have that
		$$|gw'|\ge c\mathcal N(gw')\ge c\mathcal N(g)\ge \frac{c}{8M}|g|\ge \delta|g|,$$
		as required. Hence, it suffices to prove that for every prefix $w'\le w_3w_4w_5$ we have, $|g_2w'|\ge\delta|g|$.
		Recall that by Lemma \ref{lem2}, $\mathcal N(g_2)\ge M\N(g_1)$. Hence, by Proposition \ref{Cc}, $|g_2|\ge cM\N(g_1)\ge100C\N(g_1)$. Since $||w_3||\le C\mathcal N(g_1)$ (Lemma \ref{lem3}) and $\N(g_1)\ge\N(g)$ (Lemma \ref{lem1}), for any $w'\le w_3$ we have
		$$|g_2w'|\ge|g_2|-||w'||\ge|g_2|-||w_3||\ge100C\mathcal N(g_1)-C\mathcal N(g_1)\ge99C\N(g)\ge\frac{99C}{8M}|g|\ge\delta|g|,$$
		as required. Next, we consider elements of the form $g_3w'$ where $w'\le w_4$. Recall that by Lemma \ref{lem3}, $\mathcal N(g_3)\ge (M-1)\N(g_1)$. Let $w'\le w_4$. By Lemma \ref{lem4}(1), we have
		$$\N(g_3w')\ge\N(g_3)-2\N(g_1)\ge(M-3)\N(g_1)\ge(M-3)\N(g)\ge\frac{M-3}{8M}|g|\ge\delta|g|.$$
		We also note that by Lemma \ref{lem4}(1), $\mathcal N(g_4)=\mathcal N(gw_4)\ge\frac{1}{2}||w_4||=Q|g|$. Since $$||w_5||=|g_3|\le|g|+||w_1w_2w_3||\le|g|+4M\N(g)\le |g|+\frac{4M}{c}|g|\le\frac{5M}{c}|g|\le\frac{1}{2}cQ|g|,$$
		for any prefix $w'\le w_5$, we have
		$$|g_4w'|\ge|g_4|-||w'||\ge|g_4|-||w_5||\ge c\N(g_4)-||w_5||\ge cQ|g|-\frac{1}{2}cQ|g|\ge\frac{1}{2}cQ|g|\ge\delta|g|,$$
		as required. Hence, the lemma holds in case (1).

		Case (2): $|g|\ge 8M\N(g)$.
		
		As noted above, $||w_1w_2w_3||\le 4M\N(g)$. Hence, for any $w'\le w_1w_2w_3$ we have
		$$|gw'|\ge|g|-||w'||\ge|g|-||w_1w_2w_3||\ge|g|-4M\N(g)\ge\frac{1}{2}|g|\ge\delta|g|,$$
		as required. Hence, it suffices to consider words of the form $g_3w'$ where $w'\le w_4$ and $g_4w'$ where $w'\le w_5$.
		Let $w'\le w_4$. If $||w'||\le\frac{1}{4}|g|$, then since $|g_3|\ge\frac{1}{2}|g|$, we have that $|g_3w'|\ge\frac{1}{4}|g|$, as necessary. Hence, we can assume that $||w'||>\frac{1}{4}|g|$. In that case, by Lemma \ref{lem4}(1) we have
		$$\mathcal N(g_3w')\ge\frac{1}{2}||w'||\ge\frac{1}{8}|g|.$$
		Hence, by Proposition \ref{Cc}, $|g_3w'|\ge \frac{1}{8}c|g|\ge\delta|g|$ as required.
		To finish, we note that by Lemma \ref{lem4}(1), $$\N(g_4)=\N(g_3w_4)\ge\frac{1}{2}||w_4||=Q|g|.$$
		Hence, $|g_4|\ge cQ|g|$. Since $$||w_5||=|g_3|\le |g|+||w_1w_2w_3||\le|g|+4M\N(g)\le\frac{3}{2}|g|,$$ we have that for any $w'\le w_5$,
		$$|g_4w'|\ge|g_4|-||w_5||\ge cQ|g|-\frac{3}{2}|g|\ge\delta|g|,$$
		as necessary.
		
	\end{proof}

	
\end{proof}

\begin{Theorem}
	There exist constants $\delta ,D>0$ such that the following holds. Let $g_1,g_2\in\mathcal G$ be two elements with $\N(g_1),\N(g_2)\ge 4$. Then there is a path of length at most $D(|g_1|+|g_2|)$ in the Cayley graph $\Gamma=\Cay(\mathcal G,X)$ which avoids a $\delta\min\{|g_1|,|g_2|\}$-neighborhood of the identity and which has initial vertex $g_1$ and terminal vertex $g_2$.
	
\end{Theorem}

\begin{proof}
	Let $\delta',D'$ and $Q$ be constants satisfying the conditions of  Proposition \ref{mainprop}. We set $\delta=\min\{\delta',cQ\}$ and let $D=D'+2Q$.
	By Proposition \ref{mainprop}, there is a path $w_1$ from $g_1$ to
	$x_0^{Q|g_1|}x_1^{-1}x_0^{-Q|g_1|+1}$ such that $||w_1||<D'|g|$ and such that the path avoids a $\delta'|g_1|$-neighborhood of the identity.
	Similarly, there is a path $w_2$ from $g_2$ to $x_0^{Q|g_2|}x_1^{-1}x_0^{-Q|g_2|+1}$ such that $||w_2||<D'|g_2|$ and such that the path avoids a $\delta'|g_2|$-neighborhood of the identity.
	Assume without loss of generality that $|g_1|\le |g_2|$ and let
	$$p\equiv x_0^{Q|g_1|-1}x_1x_0^{Q(|g_2|-|g_1|)}x_1^{-1}x_0^{-Q|g_2|+1},$$
	so that $p$ labels a path from $x_0^{Q|g_1|}x_1^{-1}x_0^{-Q|g_1|+1}$ to $x_0^{Q|g_2|}x_1^{-1}x_0^{-Q|g_2|+1}$.
	We let $w\equiv w_1pw_2^{-1}$. Note that $$||w||=||w_1||+||w_2||+||p||= ||w_1||+||w_2||+2Q|g_2|\le D'(|g_1|+|g_2|)+2Q|g_2|\le D(|g_1|+|g_2|).$$
	It suffices to prove that for every prefix $w'\le p$, we have
	$$|x_0^{Q|g_1|}x_1^{-1}x_0^{-Q|g_1|+1}w'|\ge\delta\min\{|g_1|,|g_2|\}.$$
	But, it is easy to see that for any prefix $w'$ of $p$, the positive part of the normal form (see \cite{CFP}) of the element $x_0^{Q|g_1|}x_1^{-1}x_0^{-Q|g_1|+1}w'\in F$ is $x_0^i$ for $i\ge Q|g_1|$. Hence, by \cite[Theorem 3]{BCS}, we have
	$$\mathcal N(x_0^{Q|g_1|}x_1^{-1}x_0^{-Q|g_1|+1}w')>Q|g_1|.$$ Hence, by Proposition \ref{Cc}, $$|x_0^{Q|g_1|}x_1^{-1}x_0^{-Q|g_1|+1}w'|\ge cQ|g_1|>\delta\min\{|g_1|,|g_2|\},$$ as $\delta<cQ$.
\end{proof}

\begin{Corollary}
	Thompson group $\mathcal G$ has linear divergence.
\end{Corollary}

\end{document}